\documentclass[11pt,reqno]{amsart}
\usepackage{graphicx,color}

\usepackage{amsmath,amsfonts,amssymb,amsthm,mathrsfs,amsopn}
\usepackage{latexsym}

\theoremstyle{plain}

\theoremstyle{plain}
\newtheorem{theorem}{Theorem} [section]

\newtheorem{lemma}[theorem]{Lemma}
\newtheorem{proposition}[theorem]{Proposition}

\theoremstyle{definition}

\theoremstyle{remark}

\numberwithin{equation}{section}

\newtheorem{remark}[theorem]{Remark}
\numberwithin{theorem}{section}
\numberwithin{equation}{section}
\numberwithin{figure}{section}

\def\mean#1{\mathchoice
         {\mathop{\kern 0.2em\vrule width 0.6em height 0.69678ex depth -0.58065ex
                 \kern -0.8em \intop}\nolimits_{\kern -0.4em#1}}%
         {\mathop{\kern 0.1em\vrule width 0.5em height 0.69678ex depth -0.60387ex
                 \kern -0.6em \intop}\nolimits_{#1}}%
         {\mathop{\kern 0.1em\vrule width 0.5em height 0.69678ex
             depth -0.60387ex
                 \kern -0.6em \intop}\nolimits_{#1}}%
         {\mathop{\kern 0.1em\vrule width 0.5em height 0.69678ex depth -0.60387ex
                 \kern -0.6em \intop}\nolimits_{#1}}}

\def\N{\mathbb N}

\def\R{\mathbb R}

\def\a{\alpha}

\def\eps{\varepsilon}

\def\H{\mathcal H}
\def\M{\mathcal M}
\def\NN{\mathcal N}
\def\F{\mathcal F}
\def\G{\mathcal G}

 \DeclareMathOperator{\dist}{dist}

\title[]{Higher integrability for minimizers\\ of the Mumford-Shah functional}

\author[G. De Philippis]{Guido De Philippis}
\address{Hausdorff Center for Mathematics,
Endenicher Allee 62, D-53115 Bonn-Germany}
\email{guido.de.philippis@hcm.uni-bonn.de}

\author[A. Figalli]{Alessio Figalli}
\address{Department of Mathematics,
The University of Texas at Austin, 1 University Station C1200,
Austin TX 78712, USA}
\email{figalli@math.utexas.edu}


\begin{document}

\begin{abstract}
We prove higher integrability for the gradient of local minimizers of the Mumford-Shah energy functional,
providing a positive answer to a conjecture of De Giorgi \cite{DeG}.
\end{abstract}

\maketitle

\section{introduction}

Free discontinuity problems are a class of variational problems which involve pairs $(u,K)$ where $K$ is some closed set and $u$ is a function which minimizes
some energy outside $K.$
One of the most famous examples is given by the Mumford-Shah energy functional, which arises in image segmentation \cite{MuS}:
given a open set $\Omega\subset \R^n$, for any $K\subset \Omega$ relatively closed 
and $u \in W^{1,2}(\Omega\setminus K)$, one defines the \textit{Mumford-Shah energy} of \((u,K)\) in \(\Omega\)  to be 
$$
MS(u, K)[\Omega]:=\int_{\Omega\setminus K} |\nabla u|^2+\H^{n-1}(K\cap \Omega).
$$
We say that the pair \((u,K)\) is a \textit{local minimizer} for the Mumford Shah energy in \(\Omega\) if, for every ball \(B=B_\varrho(x) \Subset \Omega\), 
\begin{equation*}\label{eq:minimality}
MS(u,K)[B]\le MS(v, H)[B]
\end{equation*}
for all pairs \((v,H)\) such that $H\subset \Omega$ is relatively closed,
$v \in W^{1,2}(\Omega\setminus H)$, \(K\cap (\Omega\setminus B)=H\cap (\Omega\setminus B)\), and \(u=v\) almost everywhere in \((\Omega\setminus B)\setminus K\).
We denote the set of local minimizers in \(\Omega\) by \(\M(\Omega)\). 

The existence of local minimizers is by now well-known \cite{CDGL,CL,AFP,Dav}.
In \cite{DeG}, De Giorgi formulated a series of conjectures on the properties of local minimizers.
One of them states as follows \cite[Conjecture 1]{DeG}:
\smallskip 

\textbf{Conjecture (De Giorgi):} \textit{If $(u,K)$ is a (local) minimizer of the Mumford-Shah energy inside $\Omega$, then there exists $\gamma \in (1,2)$ such that
$|\nabla u|^2 \in L^{\gamma} (\Omega'\setminus K)$ for all $\Omega'\subset\subset \Omega$.}
\smallskip

A positive answer to the above conjecture was given in \cite{DLF} when $n=2$.
The proof there strongly relies on the two-dimensional assumption, since it uses the description of minimal Caccioppoli partitions.
The aim of this note is to provide a positive answer in arbitrary dimension. Since our proof avoids any compactness argument,
our constants are potentially computable\footnote{
To be precise, the constants $\bar C$ and $\bar \gamma$ can be explicitely expressed in terms of the dimension and the constants $C_0$ and $C_\eps$ appearing in Proposition \ref{properties}.
While $C_0$ is computable, 
the constant $\eps(n)$ appearing in proposition 2.1 (iv), from which $C_\eps$ depends (see \cite{MS1,Rig} and Remark \ref{rmk:eps}), is obtained in \cite{AFPpaper} using a compactness argument.
However it seems likely that the compactness step could be avoided arguing as in \cite{SS}, but since 
this would not give any new insight to the problem, we do not investigate further this point.}.
This is our main result:
\begin{theorem}\label{thm:main}
There exist 
dimensional constants $\bar C>0$ and \(\bar \gamma=\bar \gamma(n)>1\) such that, for all   \((u,K)\in \M(B_2)\),
\begin{equation}\label{eq:main}
\int_{B_{1/2}\setminus K}|\nabla u|^{2\bar\gamma} \leq \bar C.
\end{equation}
\end{theorem}

By a simple covering/rescaling argument, one deduces the validity of the conjecture with $\gamma=\bar\gamma$.
We also remark that our result applies with trivial modifications to the ``full'' Mumford-Shah energy
\begin{equation}\label{fullMS}
MS_g(u, K)[\Omega]:=\int_{\Omega\setminus K} |\nabla u|^2+\alpha\int_{\Omega}|u-g|^2+\beta \H^{n-1}(K\cap \Omega),
\end{equation}
where $\alpha,\beta>0$, and $g\in L^2(\Omega)\cap L^\infty(\Omega)$. \\

\textit{Acknowledgements:} AF is partially supported by NSF Grant DMS-0969962.
Both authors acknowledge the support of the ERC ADG Grant GeMeThNES.
We also thank Berardo Ruffini for a careful reading of the manuscript.

\section{Preliminaries}
In the next proposition we  collect the main known properties of local minimizers that will be used in the sequel. 

\begin{proposition}\label{properties}There exists  a dimensional constant  \(C_0\) such that for all  \((u,K)\in \M(B_2)\), the following properties hold true.
\begin{enumerate}
\item[(i)] \(u\) is harmonic in \(B_2\setminus K\).
\item[(ii)] For all \(x\in B_1\) and all \(\varrho<1\)
\[
\int_{B_\varrho(x)\setminus K} |\nabla u|^2  +\H^{n-1}(K\cap B_\varrho(x))\le C_0\varrho^{n-1}.
\]
\item[(iii)] For all \(x\in K\cap B_1\) and all \(\varrho <1\),
\[
 \H^{n-1}(K\cap B_\varrho(x))\ge \varrho^{n-1}/C_0.
\]
\item[(iv)] There is a dimensional constant \(\eps(n)>0\) such that, for every \(\varepsilon\in (0,\eps(n))\), there exists \(C_\eps>0\) for which the following statement holds true:\\
For all \(x\in K\cap B_1\) and all \(\varrho <1\) there exists a \(y\in B_{\varrho/2}(x)\cap K\), a unit vector \(\bar \nu\) and a \(C^{1,1/4}\) function \(f:\R^{n-1}\to \R\)  such that 
 \begin{equation}\label{eq:graph0}
 K\cap B_{2\varrho/C_\eps}(y) =\big [y+{\rm graph}_{\bar \nu} (f)\big]\cap B_{2\varrho/C_\eps}(y),
 \end{equation}
 where
 \begin{equation}\label{eq:graph}
 {\rm graph}_{\bar \nu} (f):=\Big\{z\in \R^n:\ z\cdot \bar \nu =f\big(z-(\bar \nu \cdot z)z\big)\Big\}.
 \end{equation}
 Moreover
 \begin{equation}\label{eq:stimef}
 f(0)=0,  \qquad   \|\nabla f\|_{\infty}+\varrho ^{1/4} \|\nabla f\|_{C^{1/4}}\le C_0 \varepsilon,
 \end{equation}
 and
 \begin{equation}\label{eq:stimeu}
 \sup_{B_{2\varrho/C_\eps}(y)}\varrho|\nabla u|^2\le C_0\eps.
 \end{equation}
\end{enumerate}
\end{proposition}

\begin{proof}
Point (i) is easy. Point (ii) is well known and it can be proved by comparison, see \cite[Lemma 7.19]{AFP}. Point (iii) has been proved by Carriero, De Giorgi and Leaci in \cite{CDGL},
see also \cite[Theorem 7.21]{AFP}. Point (iv)  expresses the \emph{porosity} of the set where \(K\) is not a smooth graph.
This has been proved in \cite{AFPpaper,MS1, Rig}, see also \cite{MS2}. More precisely, in these papers it has been proved that for any fixed positive \(\eps\)
there exists a constant \(C_\eps\) such that, for all \(x\in K\cap B_1\) and \(\varrho <1\), there exists a point
\(y\in B_{\varrho}(x)\cap K\) and a ball  \(B_{r}(y)\subset B_{\varrho}(x)\), with \(r\ge 2\varrho/C_{\eps}\), such that
\begin{equation}
\label{eq:small excess}
\frac{1}{r^{n-1}}\int_{B_{r}(y)}|\nabla u(z)|^2\,dz+\frac{1}{r^{n+1}}\inf_{\nu\in S^{n-1}}\int _{K\cap B_{r}(y)}|(z-y)\cdot\nu|^2 \, d\H^{n-1}(z)\le \eps,
\end{equation}
see \cite[Theorem 1.1]{MS1}.
From this one applies the \(\eps\)-regularity theorem, \cite[Theorems 8.2 and 8.3]{AFP} to deduce \eqref{eq:graph0} and \eqref{eq:stimef}. Finally, \eqref{eq:stimeu} follows from \eqref{eq:stimef},  \eqref{eq:small excess}
and classical estimates for the Neumann problem, see for instance \cite[Theorem 7.53]{AFP}.
\end{proof}
The following simple geometric lemma will be useful:
\begin{lemma}[A geometric lemma]\label{lem:geom}Let \(G\) be closed set such that
\[
G\cap B_2={\rm graph}_{e_n}(f)
\]
for some Lipschitz function  \(f:\R^{n-1}\to\R\) satisfying
\begin{equation}\label{doccia}
f(0)=0\qquad\text{and}\qquad \|\nabla f\|_{\infty}\le \eps.
\end{equation}
Then, provided $\eps \leq 1/12$,
\[
\dist\big(x,(\overline B_{1+2\delta}\setminus B_{1+\delta})\cap G\big)\le \frac 3 2 \delta \qquad \forall\, \delta \in (0,1/2),\, x\in(\overline B_{1+\delta}\setminus B_1)\cap G.
\]
\end{lemma}

\begin{proof}
First notice that, by \eqref{doccia},
\[
\|f\|_{L^\infty(B_2)}+\|\nabla f\|_{L^\infty(B_2)}\le 2\eps.
\]
Given a point $x=(x',f(x')) \in (\overline B_{1+\delta}\setminus B_1)\cap G$, set $\alpha:=\frac{1+5\delta/4}{|x|}$
and let us consider the point $\bar x:=\bigl(\alpha x',f(\alpha x')\bigr)$.
Since $|x| \geq 1$ we have $0<(\a-1) |x| \leq 5\delta/4$, hence
\begin{align*}
|f(\alpha x')-\alpha f(x')| &\leq |f(\alpha x')-f(x')|+(\a-1)|f(x')|\\
&\leq (\a-1)
\Bigl(\|\nabla f\|_{\infty}|x'| + \|f\|_{\infty}\Bigr) \\
& \leq 2 (\a-1)\eps |x| =\frac52 \eps \delta,
\end{align*}
and 
$$
|\a x|=1+\frac{5}{4}\delta.
$$
Thus, provided $\eps \leq 1/12$ we get 
$$
\biggl||\bar x| - \biggl(1+\frac{5}{4}\delta\biggr)\biggr| \leq |\bar x -\alpha x| = |f(\alpha x')-\alpha f(x')| \leq \frac52 \eps \delta <\frac{\delta}{4}
$$
and 
$$
|\bar x - x| \leq |\bar x - \a x|+(\a-1)|x|=  |f(\alpha x')-\alpha f(x')|+(\a-1)\, |x|\leq  \frac52 \eps \delta
+\frac{5}{4}\delta< \frac32 \delta,
$$
which imply that $\bar x\in (\overline B_{1+2\delta}\setminus B_{1+\delta})\cap G$, concluding the proof.
\end{proof}

\begin{remark}\label{rmk:eps} In the sequel we will apply Proposition \ref{properties} only with $\eps:=\min\{\eps(n)/C_0,1/(12C_0)\}$, where
$\eps(n)$ and $C_0$ are as in Proposition \ref{properties}, and the factor $1/12$ comes from
Lemma \ref{lem:geom}. Hence, with this choice, also the constant \(C_{\eps}\) will be dimensional.
\end{remark}

\section{Proof of Theorem \ref{thm:main}}
Let \(M\gg 1\)  to be fixed, and for  \(h\in \N\) define the following set
\begin{equation}\label{eq:defAh}
A_h:=\big\{x\in B_2\setminus K \text{ such that } |\nabla u(x)|^2\ge M^{h+1}\big\}. 
\end{equation}
Notice that the sets \(A_h\) depend on \(M\). However, later \(M\) will be fixed to be a large  dimensional constant, so for notational simplicity we drop the dependence on \(M\).
We will use the notation \(\NN_\varrho(E)\) to denote the \(\varrho\)-neighbourhood of a set \(E\), i.e., the set of points at distance less than \(\varrho\) from \(E\).\\

The idea of the proof is the following: since $u$ is harmonic outside $K$ and the integral of $|\nabla u|^2$ over a ball of radius $r$ is controlled by $r^{n-1}$ (see Proposition \ref{properties}(ii)),
it follows by elliptic regularity that $A_h$ is contained in a $M^{-h}$-neighborhood of $K$ (Lemma \ref{lem:neigh}).
However, for the set $K$ we have a porosity estimate which tells us that inside every ball of radius $\varrho$ there is a ball of comparable radius where $|\nabla u|^2 \leq C/\varrho$
(see Proposition \ref{properties}(iv)). Hence, this implies that the size of $A_h$ is smaller than what one would get by just  using that
$A_h\subset \NN_{M^{-h}}(K)$. Indeed, by induction over $h$ we can show that $A_h$ is contained in the $M^{-h}$-neighborhood of
a set $K_h$ obtained from $K_{h-1}$ by removing the ``good balls'' where \eqref{eq:stimeu} hold (Lemma \ref{main lemma}).
Since the $\H^{n-1}$ measure of $K_h$ decays geometrically (see \eqref{eq:dec}), this allows us to obtain a stronger estimate on the size of $A_h$ which immediately implies
the higher integrability.
To make this argument rigorous we actually have to suitably localize our estimates, and for this we need to introduce some suitable sequences of radii (Lemma \ref{lem:radii}).
\smallskip

\begin{lemma}\label{lem:neigh} There exists a dimensional constant \(M_0\) such that for \(M\ge M_0\) and all \((u,K)\in \M(B_2)\) and \(R\le 1\)
\[
A_h\cap B_{R-2M^{-h}}\subset \NN_{M^{-h}}(K\cap \overline B_R)\quad \text{ for all \(h\in \N\).}
\]
\end{lemma}

\begin{proof}
Let \(x\in A_h\cap B_{R-2M^{-h}}\), \(d:=\dist(x,K)\), and \(z\in K\) a point such that \(|x-z|=d\).  If \(d>M^{-h}\) then
\[
B_{M^{-h}}(x)\cap K=\emptyset\quad \text{and}\quad \text{\(u\) is harmonic on \(B_{M^{-h}}(x)\)}.
\] 
Hence, by the definition of \(A_h\), the mean value property  for subharmonic function\footnote{
Notice that, because $u$ is harmonic, $|\nabla u|^2$ is subharmonic.
Instead, when one deals with the full functional \eqref{fullMS} (or if one wants to consider more general energy functionals than $\int |\nabla u|^2$), the mean value estimate has to be replaced by the one-sided Harnack inequality for subsolutions to uniformly elliptic equations, which in the case of minimizers of \eqref{fullMS} reads as:\[ |\nabla u(x)|^2\le C(n)\Big( \mean{B_{M^{-h}}(x)}|\nabla u|^2+\alpha^2 M^{-2h}\|g\|^2_\infty\Big). \]
}, and Proposition \ref{properties}(ii), we get
\[
M^{h+1}\le |\nabla u(x)|^2\le  \mean{B_{M^{-h}}(x)}|\nabla u|^2\le \frac{C_0}{|B_1|}M^{h},
\]
which is impossible if \(M\) is large enough. Moreover, since \(x\in B_{R-2M^{-h}}\) and \(d\le M^{-h}\) we see that \(z\in B_R\), proving the claim.
\end{proof}

\begin{lemma}[Good radii]\label{lem:radii} There are dimensional positive constants \(M_1\) and \(C_1\) such that for  \(M\ge M_1\) we can find three sequences of radii
\(\{R_h\}_{h\in \N}\), \(\{S_h\}_{h\in \N}\) and \(\{T_h\}_{h\in \N}\) for which the following properties hold true for  every \((u,K)\in \M(B_2)\).
\begin{enumerate}
\item[(i)]\(1\ge R_h\ge S_h\ge T_h \ge R_{h+1}\),
\medskip
\item[(ii)]\(R_h-R_{h+1}\le {M^{-\frac{(h+1)}{2}}}\) and \(S_h-T_{h}=T_h-R_{h+1}= 4M^{-(h+1)}\),
\medskip
\item[(iii)] \(\H^{n-1}\big(K\cap(\overline B_{S_h}\setminus\overline B_{R_{h+1}})\big)\le C_1{M^{-\frac{(h+1)}{2}}}\),
\medskip
\item[(iv)]\(R_\infty=S_\infty=T_\infty\ge 1/2\).
\end{enumerate}
\end{lemma}

\begin{proof} We set \(R_1=1\). Given \(R_h\) we show how to construct \(S_h\), \(T_h\) and \(R_{h+1}\). With no loss of generality
(up to slightly enlarge $M$) we can assume that \(\sqrt M/8\) is a natural number. Then we write
\[
\overline B_{R_h}\setminus \overline B_{R_h-M^{-\frac{(h+1)}{2}}}=\bigcup_{i=1}^{\sqrt{M^{(h+1)}}/8} \overline B_{R_h-(i-1)8M^{-(h-1)}}\setminus \overline B_{R_h-i8M^{-(h+1)}}.
\]
Being the annuli in the right hand side disjoints, there is at least an index \(\bar i\) such that
\begin{equation*}\label{treno}
\begin{split}
\H^{n-1}\Big(K\cap \overline B_{R_h-(\bar i-1)8M^{-(h-1)}}\setminus \overline B_{R_h-\bar i8M^{-(h+1)}}\Big)&\le 8M^{-\frac {(h+1)}{2}}\H^{n-1}\big(K\cap \overline B_{R_h}\setminus \overline B_{R_h-M^{-\frac{(h+1)}{2}}}\big)\\
&\le8M^{-\frac {(h+1)}{2}}\H^{n-1}(K\cap \overline B_1) \le C_1 M^{-\frac{(h+1)}{2}},
\end{split}
\end{equation*}
where in the last inequality  we have taken into account Proposition \ref{properties}(ii). Then we set
 \[
 S_h:=R_h-(\bar i-1)8M^{-(h+1)}, \quad R_{h+1}:=R_h-\bar i8M^{-(h+1)},\quad T_h:=(S_h+R_{h+1})/2.
 \] 
 Then properties (i), (ii) and (iii) trivially hold, and while (iv) follow from (ii) by choosing \(M\) large enough.
\end{proof}

\begin{lemma}\label{main lemma}
Let $C_0,\eps,C_\eps,C_1,M_1$ be as in Proposition \ref{properties} and Lemma \ref{lem:radii}, with $\eps$ as in Remark \ref{rmk:eps}.
There exist  dimensional constants \(C_2, M_2, \eta>0\), with $M_2\geq M_1$, such that,  for every \(M\ge M_2\), \((u,K)\in \M(B_2)\), and \(h\in \N\), we can find \(h\) families of disjoint  balls 
\[
\F_j=\Big\{B_{M^{-j}/C_\eps}(y_i),\quad y_i\in K,\quad i=1,\dots, N_j \Big\},\qquad j=1,\dots,h,
\]
such that
\begin{enumerate}
\item[(i)] If \(B^1, B^2\in \bigcup_{j=1}^{h}\F_j\) are distinct balls, then \(\NN_{4M^{-(h+1)}}(B^1)\cap \NN_{4M^{-(h+1)}}(B^2)=\emptyset\).
\medskip
\item[(ii)] If \(B_{M^{-j}/C_\eps}(y_i)\in \F_j\) then there is a unit vector $\nu$ and a \(C^1\) function \(f:\R^{n-1}\to \R\), with 
\[
f(0)=0\quad\text{and}\quad\|\nabla f\|_{\infty}\le \eps,
\]
such that
\[
K\cap B_{2M^{-j}/C_\eps}(y_i)=\bigl[y_i+{\rm graph }_\nu (f)\bigr]\cap B_{2M^{-j}/C_\eps}(y_i)\quad\text{and}\quad
\sup_{B_{2M^{-j}/C_\eps}(y_i)}|\nabla u|^2< M^{j+1}. 
\]

\medskip
\item[(iii)] Let  \(\{R_h\}_{h\in \N}\), \(\{S_h\}_{h\in \N}\) and \(\{T_h\}_{h\in \N}\) be the sequences of radii constructed in Lemma \ref{lem:radii} and define

\begin{equation}\label{eq:defkh}
K_{h}:=\big(K\cap \overline B_{S_h}\big)\setminus \biggl(\bigcup_{j=1}^h\bigcup_{B\in \F_j} B\biggr),
\end{equation}
and 
\begin{equation}\label{eq:defktilde}
\widetilde K_{h}:=\big(K\cap \overline B_{T_{h}}\big)\setminus  \biggl(\bigcup_{j=1}^h\bigcup_{B\in \F_j} \NN_{2M^{-(h+1)}}(B)\biggr)\subset K_{h}.
\end{equation}
Then there exists a finite set of points \(\mathcal C_h:=\{x_i\}_{i\in I_{h}}\subset \widetilde K_h\) such that
\begin{equation}\label{disjoint}
|x_j-x_k|\ge 3 M^{-(h+1)}\qquad \forall\, j,  k\in I_{h}, \, j\neq k,
\end{equation}
\begin{equation}\label{disjoint2}
\NN_{M^{-(h+1)}}( K_h\cap \overline B_{R_{h+1}})\subset \bigcup_{x_i \in \mathcal C_h} B_{8M^{-(h+1)}}(x_i). 
\end{equation}
Moreover
\begin{equation}\label{eq:dec}
\H^{n-1}(K_{h+1})\le (1-\eta) \H^{n-1}(K_{h})+C_1M^{-\frac {h+1} 2},
\end{equation}
\begin{equation}\label{eq:dec2}
\big|\NN_{M^{-{(h+1)}}}(K_h\cap \overline B_{R_{h+1}})\big|\le C_2  M^{-(h+1)}\H^{n-1}(K_{h}).
\end{equation}
\medskip
\item[(iv)] Let \(A_h\) be as in \eqref{eq:defAh}. Then
\begin{equation}\label{eq:cont}
A_{h+2}\cap B_{R_{h+2}}\subset \NN_{M^{-(h+1)}}(K_h\cap B_{R_{h+1}}).
\end{equation}
\end{enumerate}
\end{lemma}

\begin{proof}
We proceed by induction. For \(h=1\) we set \(\F_1=\emptyset\), so that \(K_1=K\cap \overline B_{S_1}\) and \(\widetilde K_1=K\cap \overline B_{T_1}\). We also choose \(\mathcal C_1\) to be a
maximal family of points at distance \(3M^{-2}\) from each other. Clearly (i), (ii), and \eqref{eq:dec} are true. The other properties can be easily obtained as in the steps below and the proof is left to the reader.

\bigskip
Assuming we have constructed \(h\) families of balls \(\{\F_j\}_{j=1}^h\) as in the statement of the Lemma, we show how to construct the family \(\F_{h+1}\).
For this, let \(\mathcal C_h=\{x_i\}_{i\in I_{h}} \subset \widetilde K_h\) be a family of points satisfying \eqref{disjoint}, and let us consider the family of disjoint  balls
$$
\G_{h+1}:=\big\{B_{M^{-(h+1)}}(x_i)\big\}_{i\in I_{h}}.
$$

\medskip
\noindent
\textit{Step 1.} We show that
\[
B_{M^{-(h+1)}}(x_i)\cap   K_h= B_{M^{-(h+1)}}(x_i)\cap  K\qquad \forall \,x_i\in \mathcal C_h.
\]
Indeed, assume by contradiction there is a point \(x\in B_{M^{-(h+1)}}(x_i)\cap (K\setminus K_h)\). First of all notice that, by Lemma \ref{lem:radii}(ii), since  \(x_i\in B_{T_h}\)
we get that \(x\in B_{S_h}\). Hence,  by the definition of \(K_h\), there is a ball \(\widetilde B\in \F_j\), \(j\le h\), such that \(x\in \widetilde B\). But then
\[
\dist(x_i,\widetilde B) \le |x-x_i|\le M^{-(h+1)}<2M^{-(h+1)},
\]
a contradiction to the fact that $x_i \in \widetilde K_h$.

\medskip
\noindent
\textit{Step 2.} We claim that there exists a positive   dimensional constant \(\eta_0\) such that, if   \(N_h\) is the cardinality of \( I_h\), then
\[
 N_h  M^{-(h+1)(n-1)}\ge \eta_0 \H^{n-1}(K_h\cap \overline B_{R_{h+1}})
\]
Indeed, by \eqref{disjoint2} and Proposition \ref{properties}(ii),
\[
\begin{split}
 \H^{n-1}(K_h\cap \overline B_{R_{h+1}})&=  \H^{n-1}\Big(K_h\cap \overline B_{R_{h+1}}\cap \bigcup_{x_i \in \mathcal C_h} B_{8M^{-(h+1)}}(x_i)\Big)\\
 &\le C_0N_h\big(8M^{-(h+1)}\big)^{n-1}= \frac {1}{\eta_0} N_h M^{-(h+1)(n-1)},
\end{split}
\]
where $\eta_0:=1/(C_08^{n-1})$.

\medskip
\noindent
\textit{Step 3.} By Proposition \ref{properties}(iv) and Remark \ref{rmk:eps}, for every ball \(B_{M^{-(h+1)}}(x_i)\in \G_{h+1}\) there exists a ball
\begin{equation}
\label{eq:inclusion}
B_{M^{-(h+1)}/C_\eps}(y_i)\subset B_{M^{-(h+1)}}(x_i)
\end{equation}
such that
\[
\sup_{B_{2M^{-(h+1)}/C_\eps}(y_i)}|\nabla u|^2\le \eps M^{h+1}< M^{h+2}.
\]
and
\[
K\cap B_{2M^{-(h+1)}/C_\eps}(y_i)=\bigl[y_i+{\rm graph }_\nu (f)\bigr]\cap B_{2M^{-(h+1)}/C_\eps}(y_i),
\]
for some unit vector \(\nu\) and some \(C^1\) function \(f\) such that
\[
f(0)=0\quad\text{and}\quad\|\nabla f\|_{\infty}\le \eps.
\]
 We define 
\[
\F_{h+1}:=\big\{ B_{M^{-(h+1)}/C_\eps}(y_i)\big\}_{i\in I_h}.
\]
In this way property (ii) in the statement of the lemma is satisfied.
Moreover, since the balls $\{B_{3M^{-(h+1)}/2}(x_i)\}_{i \in I_h}$ are disjoint (because $|x_j-x_k|\geq 3M^{-(h+1)}$) and do not intersect
$$
\bigcup_{j=1}^h\bigcup_{B\in \F_j} \NN_{\frac{M^{-(h+1)}}{2}}(B)
$$
it follows from \eqref{eq:inclusion} that also property (i) is satisfied provided we choose $M$ sufficiently large. \\
We  define \(K_{h+1}\) and \(\widetilde K_{h+1}\) as in the statement of the lemma and we take \(\mathcal C_{h+1}=\{x_i\}_{i\in I_{h+1}}\) a maximal sets of points in \(\widetilde K_{h+1}\) satisfying
\[
|x_j-x_k|\ge 3M^{-(h+2)}\qquad \forall\,j\ne k.
\]

\medskip
\noindent
\textit{Step  4.} The set of  points \(\mathcal C_{h+1}\) defined in the previous step  satisfies  by construction \eqref{disjoint}.
We now prove it also satisfies \eqref{disjoint2}. For this, let \(x\in \NN_{M^{-(h+2)}}(K_{h+1}\cap \overline B_{R_{h+2}})\) and let \(\bar x \in K_{h+1}\cap B_{R_{h+2}}\) be such that
\[
|x-\bar x|=\dist\big(x,K_{h+1}\cap \overline B_{R_{h+2}}\big)\le M^{{-(h+2)}}.
\]
In case \(\bar x\in \widetilde K_{h+1}\), by maximality there exists a point \(x_i\in \mathcal C_{h+1}\) such that \(|\bar x-x_i|\le 3M^{-(h+2)}\), hence \(x\in B_{5M^{-(h+2)}}(x_i)\) and we are done.
So, let us assume that
\[
\bar x\in \left( K_{h+1}\cap \overline B_{R_{h+2}}\right)\setminus \widetilde K_{h+1}.
\]
In this case, by the definition of $K_{h+1}$ and $\widetilde K_{h+1}$, there exists a ball \(\widetilde B\in \bigcup_{j=1}^{h+1} \F_j\) such that
\[
x\in K\cap \mathcal \NN_{2M^{-(h+2)}}(\widetilde B)\setminus \widetilde B.
\]
Thanks to property (ii) we can apply (a scaled version of) Lemma \ref{lem:geom} to find a point 
\[
y\in K\cap \NN_{4M^{-(h+2)}}(\widetilde B)\setminus \NN_{2M^{-(h+2)}}(\widetilde B)
 \]
such that
\[
|\bar x-y|\le 3M^{-(h+2)}.
\]
Since \(\bar x\in \overline B_{R_{h+2}}\) and $T_{h+1}=R_{h+2}+ 4M^{-(h+2)}$
\[
y\in \left(K\cap B_{T_{h+1}} \cap \NN_{4M^{-(h+2)}}(\widetilde B)\right)\setminus \NN_{2M^{-(h+2)}}(\widetilde B) \subset \widetilde K_{h+1},
\]
where the last inclusion follows by property (i) and the definition of \(\widetilde K_h\).
Again by maximality, there exists a point \(x_i \in \mathcal C_{h+1}\) such that $|y-x_i|\leq 3M^{-(h+2)}$, hence
\[
|x_i-x|\le |x_i-y|+|y-\bar x|+|\bar x -x|\le 7M^{-(h+2)},
\] 
which completes the proof of \eqref{disjoint2}.

\medskip
\noindent
\textit{Step 5.} We prove \eqref{eq:dec}. Notice that, being the balls in \(\F_{h+1}\) disjoint, 
thanks to Step 1,
the density estimates in  Proposition \ref{properties}(iii),  Step 2, and choosing \(\eta:=\eta_0/C_0^n\) we get
\[
\begin{split}
\H^{n-1}(K_{h+1})&\le \H^{n-1}\Big(K_h\setminus \bigcup_{i\in I_h} B_{M^{-(h+1)}/C_\eps}(y_i)\Big)\\
&=\H^{n-1}(K_h )-\sum_{i\in I_h} \H^{n-1}\bigl(K_h\cap B_{M^{-(h+1)}/C_\eps}(y_i)\bigr)\\
&\le \H^{n-1}(K_h)-\frac {N_h}{C_0^n}M^{-(h+1)(n-1)}\\
&\le \H^{n-1}(K_h)-\frac{\eta_0}{C^n_0} \H^{n-1}( K_h\cap \overline B_{R_{h+1}})\\
&= (1-\eta)\H^{n-1}(K_h)+\eta\bigl[\H^{n-1}(K_h) - \H^{n-1}( K_h\cap \overline B_{R_{h+1}})\bigr]\\
&\le (1-\eta)\H^{n-1}(K_h)+\H^{n-1}\big(K\cap \overline B_{S_h}\setminus \overline B_{R_{h+1}} \big)\\
&\le (1-\eta)\H^{n-1}(K_h)+C_1M^{-\frac{(h+1)}{2}},
\end{split}
\]
where in the last step we used Lemma \ref{lem:radii}(iii).

\medskip
\noindent
\textit{Step 6.} We prove \eqref{eq:dec2}. By \eqref{disjoint2},
\[
\NN_{M^{-(h+2)}}(K_{h+1}\cap \overline B_{R_{h+2}})\subset \bigcup_{x_i\in C_{h+1}}B_{8M^{-(h+2)}}(x_i),
\]
hence, denoting with \(N_{h+1}\) the cardinality of \(I_{h+1}\),
\begin{equation}\label{freddo}
\big|\NN_{M^{-(h+2)}}(K_{h+1}\cap \overline B_{R_{h+2}})\big|\le 8^{n}M^{-(h+2)} N_{h+1}M^{-(h+2)(n-1)}.
\end{equation}
Also, by Step 1 with \(h\) replaced by \(h+1\),
\[
B_{M^{-(h+2)}}(x_i)\cap   K_{h+1}= B_{M^{-(h+2)}}(x_i)\cap  K \qquad \forall\, x_i \in\mathcal C_{h+1},
\]
hence by the density estimates in Proposition \ref{properties}(iii),
\[
M^{-(h+2)(n-1)}\le C_0 \H^{n-1}\big(K_{h+1}\cap B_{M^{-(h+2)}}(x_i)\big).
\]
The above equation and \eqref{freddo}, together with the disjointness of the balls \(\big\{B_{M^{-(h+2)}}(x_i)\big\}_{i\in I_{h+1}}\) imply
\[
\big|\NN_{M^{-(h+2)}}(K_{h+1}\cap \overline B_{R_{h+2}})\big|\le C_08^{n}M^{-(h+2)} \sum_{i\in I_{h+1}}\H^{n-1}\big(K_{h+1}\cap B_{M^{-(h+2)}}(x_i)\big)\le C_2 M^{-(h+2)}\H^{n-1}(K_{h+1}).
\]
\medskip
\noindent
\textit{Step 7.} We are left to show point (iv). Let \(x\in A_{h+3}\cap B_{R_{h+3}}\). By Lemma \ref{lem:radii} \(R_{h+2}-R_{h+3}\ge 8 M^{-(h+3)}\), hence, by Lemma \ref{lem:neigh},
\[
A_{h+3}\cap B_{R_{h+3}} \subset \NN_{M^{-(h+3)}}(K\cap \overline B_{R_{h+2}})\subset \NN_{M^{-(h+2)}}(K\cap \overline B_{R_{h+2}}).
\]
Let \(\bar x\in K\cap \overline B_{R_{h+2}} \) a point realizing the distance and assume by contradiction that \(\bar x\in K\setminus K_{h+1}\).
By the definition of \(K_{h+1}\) and since \(R_{h+2}\le S_{h+1}\), this means that there is a ball \(\widetilde B\in \bigcup_{j=1}^{h+1}\F_j\) such that \(\bar x\in \widetilde B \).
Since $|\bar x-x| \leq M^{-(h+2)}$ and the radius of \(\widetilde B\) is at least  \(M^{-(h+1)}/C_\eps\),
we can choose \(M\) large enough so that \(x\in 2 \widetilde B\). But then, by property (ii) of the statement,
\[
|\nabla u (x)|^2< M^{h+2},
\]
a contradiction to the fact that \(x\in A_{h+3}\).
\end{proof}

We are now in the position to prove Theorem \ref{thm:main}:
\begin{proof}[Proof of Theorem \ref{thm:main}] Iterating \eqref{eq:dec} we obtain
\begin{equation}\label{dec2}
\H^{n-1}(K_h)\le C_1\sum_{i=0}^{h}(1-\eta)^{h-i}M^{-\frac{i}{2}}\le C_1 h\max \big\{(1-\eta)^h,M^{-h/2}\big\}.
\end{equation}
We now fix \(M:=M_2\) where \(M_2\) is the constant appearing in Lemma \ref{main lemma}, and choose \(\alpha \in (0,1/4)\) such that \((1-\eta)\le M^{-2\alpha}\). In this way,
since \(2\alpha<1/2\) it follows from \eqref{dec2} that
$$
\H^{n-1}(K_h)\le C_1 h M^{-2\alpha h}.
$$
Hence, by \eqref{eq:cont}, \eqref{eq:dec2},   and the above equation, we obtain
$$
|A_{h+2}\cap B_{R_{h+2}}|\le  \big|\NN_{M^{-(h+1)}}(K_h\cap \overline B_{R_{h+1}})\big|\le C_1C_2 h M^{-h(1+2\alpha)}\qquad \forall\,h\ge 1,
$$
so Lemma \ref{lem:radii}(iv) and the definition of \(A_h\) (see \eqref{eq:defAh}) finally give
\begin{equation}
\label{eq:decay}
\bigl|\{x\in B_{1/2}\setminus K: |\nabla u|^2(x)\ge M^{h}\}\bigr|\le C_1C_2 M^{2+4\alpha} h M^{ -h(1+2\alpha)}\qquad \forall\,h\ge 3.
\end{equation}
Since
\[
\begin{split}
\int_{B_{1/2}\setminus K}|\nabla u |^{2\gamma}&= \gamma\int_0^\infty t^{\gamma-1}|(B_{1/2}\setminus K)\cap \{|\nabla u|^2 \ge t\}| \,dt\\
&\le M^{\gamma} \sum_{h=0}^\infty M^{h\gamma}\bigl | (B_{1/2}\setminus K)\cap \{|\nabla u|^2\ge M^{h}\}\bigr|,
\end{split}
\]
\eqref{eq:decay} implies the validity of
\eqref{eq:main} with, for instance, \(\bar\gamma=1+\alpha\).
\end{proof}

\end{document}